\documentclass[11pt]{amsart}

\usepackage[T1]{fontenc}
\usepackage[utf8]{inputenc}
\usepackage[margin=1.15in]{geometry}
\usepackage{amsmath,amssymb,amsthm,mathtools}
\usepackage{enumitem}

\usepackage[colorlinks=true,linkcolor=blue,citecolor=blue,urlcolor=blue]{hyperref}

\newtheorem{theorem}{Theorem}[section]

\newtheorem{proposition}[theorem]{Proposition}
\newtheorem{lemma}[theorem]{Lemma}
\newtheorem{corollary}[theorem]{Corollary}
\theoremstyle{remark}
\newtheorem{remark}[theorem]{Remark}

\numberwithin{equation}{section}

\newcommand{\F}{\mathbb F}
\newcommand{\Sym}{\operatorname{Sym}}
\newcommand{\GL}{\operatorname{GL}}
\newcommand{\diag}{\operatorname{diag}}

\newcommand{\rank}{\operatorname{rank}}

\DeclareMathOperator{\probb}{\mathbb P}
\DeclareMathOperator{\Ir}{Ir}

\title[Normal sign matrices]{Sharp exponential asymptotics for normal sign matrices}

\author{Caden Young}
\address{University of Tulsa, Tulsa, Oklahoma, USA}

\date{June 2026}

\subjclass[2020]{Primary 15B52; Secondary 15A27, 15B33, 05A16, 60B20}
\keywords{random sign matrix, normal matrix, Rademacher matrix, finite field, commuting matrices, symmetric matrices}
\begin{document}

\begin{abstract}
Let $M_n$ be an $n\times n$ random matrix whose entries are independent Rademacher random variables, and put $N=\binom n2$.  We prove
\[
        \probb(M_nM_n^T=M_n^TM_n)=2^{-N+O(n)}.
\]
This gives the sharp exponential order for the probability that a random sign matrix is normal.  The lower bound is supplied by symmetric sign matrices.  We also record the immediate consequence that random $0$-$1$ matrices have the same sharp exponential normality probability.  The proof of the matching upper bound is combinatorial: after conditioning on the symmetric/skew-symmetric type pattern of the off-diagonal entries, a mod-$4$ reduction gives a system of linear equations over $\F_2$; a rank-duality argument converts the sum over type patterns into a count of commuting symmetric pairs over $\F_2$; and this count is bounded by summing over rational canonical types, using balanced symmetric bilinear forms and the standard finite-field centralizer formula.
\end{abstract}
\maketitle

\section{Introduction}

An $n\times n$ real matrix $M$ is normal if
\[
        MM^T=M^TM.
\]
Let $M_n$ be chosen uniformly from $\{\pm1\}^{n\times n}$, equivalently with independent Rademacher entries.  We study
\[
        p_n:=\probb(M_nM_n^T=M_n^TM_n).
\]
Since every symmetric sign matrix is normal, and since a symmetric sign matrix is determined by its $\binom n2$ upper-triangular off-diagonal entries and its $n$ diagonal entries, one immediately has
\[
        p_n\geq \frac{2^{\binom n2+n}}{2^{n^2}}=2^{-\binom n2}.
\]
Deneanu and Vu proved the first general exponential bounds for this probability, showing that
\[
        2^{-(1/2+o(1))n^2}\leq p_n\leq 2^{-(0.302+o(1))n^2},
\]
and conjectured that the lower-bound exponent is sharp \cite{DeneanuVu2019}.  We prove this conjecture in the following stronger form.

\begin{theorem}\label{thm:main}
Let $M_n$ be uniformly random in $\{\pm1\}^{n\times n}$, and put $N=\binom n2$.  Then
\[
        2^{-N}\leq \probb(M_nM_n^T=M_n^TM_n)\leq 2^{-N+O(n)}.
\]
Equivalently,
\[
        \probb(M_nM_n^T=M_n^TM_n)=2^{-\binom n2+O(n)}=2^{-(1/2+o(1))n^2}.
\]
\end{theorem}

The point of Theorem~\ref{thm:main} is that the elementary symmetric-matrix construction gives the correct exponential scale, up to only a linear error in the base-$2$ exponent.  A short consequence is that the same exponent holds for random $0$-$1$ matrices.  This is recorded in Section~\ref{sec:zero-one}; the relation is one-way under the affine change of variables $M=2U-J$, but it is enough to transfer the upper bound.

\subsection{Context}
Questions about rare algebraic properties of discrete random matrices are a central theme in combinatorial random matrix theory.  The most developed example is the singularity problem for Bernoulli matrices, where a sequence of increasingly sharp anti-concentration arguments culminated in the asymptotic formula of Tikhomirov \cite{Tikhomirov2020}, building on work such as Tao and Vu \cite{TaoVu2007}.  Other discrete spectral questions include the simple-spectrum problem for random symmetric matrices \cite{TaoVu2017}.  The present problem is different in flavor: normality imposes many quadratic equations, but the special $\{\pm1\}$ structure allows a mod-$4$ linearization after the right decomposition.
\subsection{Proof overview}
We write
\[
        A=\frac{M+M^T}{2},\qquad B=\frac{M-M^T}{2}.
\]
Then $A^T=A$, $B^T=-B$, and
\[
        MM^T-M^TM=2(BA-AB),
\]
so normality is equivalent to $[A,B]=0$.  Each off-diagonal pair $\{i,j\}$ is either symmetric, meaning $m_{ij}=m_{ji}$, or skew-symmetric, meaning $m_{ij}=-m_{ji}$.  This gives a type pattern
\[
        t=(t_{ij})_{1\leq i<j\leq n}\in \F_2^{\binom n2}.
\]
Once $t$ is fixed, the remaining signs are encoded by $N+n$ binary variables.

For each fixed type pattern, the equations $[A,B]_{ij}=0$ imply affine linear equations over $\F_2$ obtained by reducing the integer equation modulo $4$.  If $r(t)$ denotes the rank of the resulting homogeneous forms, the number of normal assignments with type pattern $t$ is at most $2^{N+n-r(t)}$.  Therefore
\[
        \#\{M\in\{\pm1\}^{n\times n}:MM^T=M^TM\}
        \leq 2^{N+n}\sum_t2^{-r(t)}.
\]

The main combinatorial step is a rank-duality transformation.  It converts the sum over type patterns into a count of pairs of symmetric matrices over $\F_2$ that commute:
\[
        \#\{M\in\{\pm1\}^{n\times n}:MM^T=M^TM\}
        \leq 2^n C_n,
\]
where
\[
        C_n:=\#\{(X,Y)\in\Sym_n(\F_2)^2:XY=YX\}.
\]
It remains to prove
\[
        C_n\leq 2^{\binom n2+O(n)}.
\]
This final estimate is obtained by grouping $X$ according to rational canonical type.  If a type is denoted by $\Lambda$, and if $c(\Lambda)$ is the dimension of the full matrix centralizer, then the space of symmetric matrices commuting with a symmetric representative has dimension
\[
        s(\Lambda)=\frac{n+c(\Lambda)}2.
\]
The same dimension appears as the dimension of the space of symmetric bilinear forms balanced with respect to a representative of type $\Lambda$.  We use this to bound the number of symmetric representatives of type $\Lambda$ by comparing the conjugating matrices that produce symmetric representatives with nondegenerate balanced forms and then dividing by the $\GL_n(\F_2)$-centralizer of the representative.  This is close in spirit to the rational-canonical-form summations of Feit--Fine and Fulman, and to the bilinear-form viewpoint used in work on symmetric nilpotent matrices \cite{FeitFine1960,Fulman2002,BrouwerGowSheekey2014}.  The remaining factor, the number of rational canonical types of size $n$, is only $2^{O(n)}$ \cite{Green1955,Fulman2002,Stong1988}.

\section{Decomposition and type patterns}

Throughout the paper let
\[
        N=\binom n2.
\]
Given a sign matrix $M=(m_{ij})$, define
\[
        A=\frac{M+M^T}{2},\qquad B=\frac{M-M^T}{2}.
\]
Then $A^T=A$, $B^T=-B$, and $M=A+B$, $M^T=A-B$.  Hence
\begin{align*}
MM^T-M^TM
&=(A+B)(A-B)-(A-B)(A+B)  \\
&=2(BA-AB).
\end{align*}
Therefore
\begin{equation}\label{eq:normal-commutator}
        MM^T=M^TM\qquad\Longleftrightarrow\qquad [A,B]=0.
\end{equation}

For each $i<j$, define $t_{ij}\in\F_2$ by
\[
        t_{ij}=0\Longleftrightarrow m_{ij}=m_{ji},
        \qquad
        t_{ij}=1\Longleftrightarrow m_{ij}=-m_{ji}.
\]
Also define signs $x_{ij}\in\F_2$, $i<j$, and $z_i\in\F_2$ by
\[
        m_{ij}=(-1)^{x_{ij}},\qquad m_{ii}=(-1)^{z_i}.
\]
Then
\[
        m_{ji}=(-1)^{x_{ij}+t_{ij}}.
\]
If $t_{ij}=0$, then
\[
        A_{ij}=A_{ji}=(-1)^{x_{ij}},\qquad B_{ij}=B_{ji}=0,
\]
while if $t_{ij}=1$, then
\[
        A_{ij}=A_{ji}=0,\qquad B_{ij}=(-1)^{x_{ij}},\qquad B_{ji}=-(-1)^{x_{ij}}.
\]
The diagonal entries are $A_{ii}=(-1)^{z_i}$ and $B_{ii}=0$.  Thus every sign matrix is encoded by
\[
        (t,x,z)\in\F_2^N\times\F_2^N\times\F_2^n.
\]
For distinct $a,b$, we write
\[
        t_{ab}:=t_{\min(a,b),\max(a,b)},\qquad
        x_{ab}:=x_{\min(a,b),\max(a,b)}.
\]

\section{The mod-$4$ linear filter}

Fix a type pattern $t$.  We derive necessary linear conditions over $\F_2$ for the remaining variables $(x,z)$.

For $i<j$, expand
\[
        [A,B]_{ij}=\sum_{k=1}^n(A_{ik}B_{kj}-B_{ik}A_{kj}).
\]
Separating $k=i,j$ gives
\begin{equation}\label{eq:comm-expanded}
        [A,B]_{ij}=\sum_{k\ne i,j}(A_{ik}B_{kj}-B_{ik}A_{kj})+(A_{ii}-A_{jj})B_{ij}.
\end{equation}
For $k\ne i,j$, the summand in \eqref{eq:comm-expanded} vanishes unless
\[
        t_{ik}\ne t_{jk}.
\]
Indeed, if both pairs are symmetric, the relevant $B$-entries vanish; if both pairs are skew-symmetric, the relevant $A$-entries vanish.

Let
\[
        K_{ij}:=\{k\ne i,j:t_{ik}\ne t_{jk}\}=\{k_1,\ldots,k_q\}.
\]
For every $k\in K_{ij}$, the corresponding summand has the form
\[
        \epsilon_{ijk}(-1)^{x_{ik}+x_{jk}},
\]
where $\epsilon_{ijk}\in\{\pm1\}$ depends only on $t$ and on the order of the indices.  Write $\epsilon_{ijk_\alpha}=(-1)^{a_\alpha}$ with $a_\alpha\in\F_2$, and set
\[
        y_\alpha=x_{ik_\alpha}+x_{jk_\alpha}\in\F_2.
\]
The diagonal contribution is
\[
        (A_{ii}-A_{jj})B_{ij}=t_{ij}\bigl((-1)^{z_i}-(-1)^{z_j}\bigr)(-1)^{x_{ij}}.
\]
Thus $[A,B]_{ij}=0$ implies
\begin{equation}\label{eq:integer-normality}
        0=\sum_{\alpha=1}^q(-1)^{y_\alpha+a_\alpha}
        +t_{ij}\bigl((-1)^{z_i}-(-1)^{z_j}\bigr)(-1)^{x_{ij}}.
\end{equation}
Reducing \eqref{eq:integer-normality} modulo $4$ and using
\[
        (-1)^u\equiv 1-2u\pmod 4\qquad (u\in\F_2),
\]
we obtain
\[
        q-2\sum_{\alpha=1}^q(y_\alpha+a_\alpha)+2t_{ij}(z_i+z_j)\equiv0\pmod4.
\]
If $q$ is odd, this congruence has no solutions.  If $q$ is even, division by $2$ modulo $2$ gives an affine linear equation in $(x,z)$.  Its homogeneous part is
\begin{equation}\label{eq:linear-form}
        \ell^t_{ij}(x,z)
        :=\sum_{\substack{k\ne i,j\\ t_{ik}\ne t_{jk}}}(x_{ik}+x_{jk})+t_{ij}(z_i+z_j).
\end{equation}
For fixed $t$, define
\[
        r(t):=\rank_{\F_2}\{\ell^t_{ij}:1\leq i<j\leq n\}.
\]

\begin{lemma}\label{lem:linear-filter}
For fixed $t$, the number of normal sign matrices with type pattern $t$ is at most
\[
        2^{N+n-r(t)}.
\]
Consequently,
\begin{equation}\label{eq:sum-rank-bound}
        \#\{M\in\{\pm1\}^{n\times n}:MM^T=M^TM\}
        \leq 2^{N+n}\sum_t2^{-r(t)}.
\end{equation}
\end{lemma}

\begin{proof}
For fixed $t$, the remaining variables $(x,z)$ lie in $\F_2^{N+n}$.  The congruences obtained above are necessary conditions for normality.  If any of the associated integers $q=|K_{ij}|$ is odd, then there are no normal assignments satisfying that equation, and the claimed upper bound is trivial.  Otherwise, the normal assignments are contained in an affine linear system whose homogeneous row space is spanned by the forms $\ell^t_{ij}$.  The dimension of this affine system is at most $N+n-r(t)$, so it contains at most $2^{N+n-r(t)}$ points.  Summing over $t\in\F_2^N$ gives \eqref{eq:sum-rank-bound}.
\end{proof}

\section{Rank duality and commuting symmetric pairs}

Let $R_t$ be the $N\times(N+n)$ matrix over $\F_2$ whose rows are the coefficient vectors of the forms $\ell^t_{ij}$.  Then $r(t)=\rank R_t$.  By rank-nullity for the left nullspace,
\[
        2^{-r(t)}=2^{-N}\#\{h\in\F_2^N:h^TR_t=0\}.
\]
Therefore \eqref{eq:sum-rank-bound} gives
\begin{equation}\label{eq:rank-duality-first}
        \#\{M\in\{\pm1\}^{n\times n}:MM^T=M^TM\}
        \leq 2^n\#\{(t,h):h^TR_t=0\}.
\end{equation}

We now interpret the condition $h^TR_t=0$ as a matrix-commutation condition.  Identify $t$ and $h$ with symmetric zero-diagonal matrices $T,H\in\Sym_n(\F_2)$ by
\[
        T_{ab}=T_{ba}=t_{ab},\qquad H_{ab}=H_{ba}=h_{ab},\qquad T_{aa}=H_{aa}=0.
\]
Define $L_H\in\Sym_n(\F_2)$ by
\[
        (L_H)_{aa}=\sum_{c\ne a}h_{ac},\qquad (L_H)_{ab}=h_{ab}\quad(a\ne b).
\]

\begin{lemma}\label{lem:rank-duality}
For fixed $t$ and $h$, the condition $h^TR_t=0$ is equivalent to
\[
        \diag(HT)=0
        \qquad\text{and}\qquad
        L_HT=TL_H.
\]
\end{lemma}

\begin{proof}
The condition $h^TR_t=0$ says that the linear form
\[
        \sum_{i<j}h_{ij}\ell^t_{ij}
\]
is zero.  Equivalently, the coefficient of every variable $z_a$ and $x_{ab}$ is zero.

First consider $z_a$.  The variable $z_a$ appears in $\ell^t_{ab}$ with coefficient $t_{ab}$.  Hence its total coefficient is
\[
        \sum_{b\ne a}h_{ab}t_{ab}=(HT)_{aa}.
\]
Thus all $z$-coefficients vanish exactly when $\diag(HT)=0$.

Now fix $a<b$.  The variable $x_{ab}$ can appear only in rows indexed by pairs $\{a,c\}$ or $\{b,c\}$, where $c\ne a,b$.  From the row $\{a,c\}$ it appears with coefficient $t_{ab}+t_{bc}$, and from the row $\{b,c\}$ it appears with coefficient $t_{ac}+t_{ab}$.  Therefore the coefficient of $x_{ab}$ in the full row combination is
\begin{equation}\label{eq:x-coeff}
        \sum_{c\ne a,b}h_{ac}(t_{ab}+t_{bc})
        +\sum_{c\ne a,b}h_{bc}(t_{ac}+t_{ab}).
\end{equation}
Let $d_a=\sum_{c\ne a}h_{ac}$.  The $(a,b)$ entry of $L_HT+TL_H$ is
\[
        \sum_{c\ne a,b}h_{ac}t_{bc}+\sum_{c\ne a,b}t_{ac}h_{bc}+t_{ab}(d_a+d_b),
\]
which is exactly \eqref{eq:x-coeff}.  Hence all $x$-coefficients vanish exactly when $L_HT+TL_H=0$.  Since the field has characteristic $2$, this is equivalent to $L_HT=TL_H$.
\end{proof}

\begin{corollary}\label{cor:commuting-pairs-reduction}
Let
\[
        C_n:=\#\{(X,Y)\in\Sym_n(\F_2)^2:XY=YX\}.
\]
Then
\[
        \#\{M\in\{\pm1\}^{n\times n}:MM^T=M^TM\}\leq 2^n C_n.
\]
\end{corollary}

\begin{proof}
By Lemma~\ref{lem:rank-duality}, dropping the condition $\diag(HT)=0$ can only enlarge the set counted in \eqref{eq:rank-duality-first}.  The map
\[
        (T,H)\longmapsto (T,L_H)
\]
is injective, since $H$ is recovered from the off-diagonal entries of $L_H$.  Therefore
\[
        \#\{(t,h):h^TR_t=0\}
        \leq \#\{(X,Y)\in\Sym_n(\F_2)^2:XY=YX\}=C_n.
\]
Combining this with \eqref{eq:rank-duality-first} proves the claim.
\end{proof}

\section{Commuting symmetric pairs over $\F_2$}\label{sec:commuting-pairs}

It remains to prove the following estimate.

\begin{proposition}\label{prop:commuting-pairs}
One has
\[
        C_n=\#\{(X,Y)\in\Sym_n(\F_2)^2:XY=YX\}\leq 2^{N+O(n)}.
\]
\end{proposition}

The proof in this section is written in rational-canonical-form language.  It follows the same general counting philosophy as Feit--Fine's enumeration of commuting matrix pairs over finite fields \cite{FeitFine1960}: fix the rational canonical type of the first matrix, count possible second matrices by a centralizer dimension, and then sum over types.  The extra point here is that the first matrix and the second matrix are both required to be symmetric.  We handle this by translating the symmetric centralizer condition into a statement about symmetric bilinear forms balanced with respect to the relevant $\F_2[t]$-module; compare the bilinear-form viewpoint in Brouwer--Gow--Sheekey \cite{BrouwerGowSheekey2014}.

\subsection{Rational canonical types}

Let $\Ir(\F_2[t])$ denote the set of monic Ireducible polynomials over $\F_2$.  A rational canonical type of size $n$ is a finite-support collection
\[
        \Lambda=(\lambda_\phi)_{\phi\in\Ir(\F_2[t])}
\]
of partitions satisfying
\[
        \sum_\phi d_\phi |\lambda_\phi|=n,
        \qquad d_\phi:=\deg \phi.
\]
The matrices in $M_n(\F_2)$ of type $\Lambda$ form one ordinary similarity class.  If $T$ has type $\Lambda$, then the dimension of its full matrix centralizer is
\begin{equation}\label{eq:centralizer-dimension-lambda}
        c(\Lambda)=\dim_{\F_2}\{Z\in M_n(\F_2):ZT=TZ\}
        =\sum_\phi d_\phi\sum_{j\geq1}(\lambda'_{\phi,j})^2,
\end{equation}
where $\lambda'_{\phi}$ denotes the conjugate partition.  This is the standard centralizer-dimension formula for rational canonical form; see, for instance, Green \cite{Green1955} or Fulman's finite-field survey \cite{Fulman2002}.

For a partition $\lambda$, write $m_i(\lambda)$ for the number of parts of $\lambda$ equal to $i$.  We shall use the standard formula
\begin{equation}\label{eq:gl-centralizer-formula}
        |C_{\GL_n}(T)|
        =2^{c(\Lambda)}
        \prod_\phi\prod_{i\geq1}\prod_{j=1}^{m_i(\lambda_\phi)}
        \bigl(1-2^{-jd_\phi}\bigr),
\end{equation}
where
\[
        C_{\GL_n}(T)=\{g\in\GL_n(\F_2):gT=Tg\}.
\]
This form of the formula is also standard in the rational-canonical-form enumeration literature \cite{Fulman2002}.

\subsection{Symmetric balanced forms and the symmetric centralizer}

Let $V=\F_2^n$.  For a linear operator $T$ on $V$, call a bilinear form $B:V\times V\to\F_2$ \emph{$T$-balanced} if
\[
        B(Tu,v)=B(u,Tv)\qquad(u,v\in V).
\]

\begin{lemma}\label{lem:balanced-dimension-lambda}
Let $T$ have rational canonical type $\Lambda$.  The vector space of symmetric $T$-balanced bilinear forms on $V$ has dimension
\[
        s(\Lambda):=\frac{n+c(\Lambda)}2.
\]
Consequently, if $X=X^T$ has type $\Lambda$, then
\[
        \dim_{\F_2}\{Y\in\Sym_n(\F_2):YX=XY\}=s(\Lambda).
\]
\end{lemma}

\begin{proof}
Make $V$ into an $\F_2[t]$-module by letting $t$ act as $T$.  Decompose $V$ into its rational canonical summands:
\[
        V\simeq\bigoplus_\phi V_\phi,
        \qquad
        V_\phi\simeq\bigoplus_{a\in\lambda_\phi}\F_2[t]/(\phi^a).
\]
Balanced forms between distinct primary parts vanish.  More precisely, if
\[
        M=\F_2[t]/(\phi^a),\qquad M'=\F_2[t]/(\psi^b),
\]
then the space of balanced bilinear forms $M\times M'\to\F_2$ has dimension
\[
        \begin{cases}
        d_\phi\min(a,b),&\phi=\psi,\\
        0,&\phi\ne\psi.
        \end{cases}
\]
Indeed, such a form is determined by the functional $y\mapsto B(1,y)$ on $M'$, with the relation that this functional vanish on $\phi^aM'$.

Fix $\phi$ and write $\lambda_\phi=(\lambda_1,\ldots,\lambda_r)$ and $d=d_\phi$.  The contribution of the $\phi$-primary part to the full centralizer dimension is
\[
        c_\phi
        =d\sum_{i,j=1}^r\min(\lambda_i,\lambda_j)
        =d\sum_{j\geq1}(\lambda'_j)^2.
\]
The contribution of this primary part to the dimension of symmetric balanced forms is
\[
        s_\phi
        =d\sum_i\lambda_i+d\sum_{i<j}\min(\lambda_i,\lambda_j).
\]
The first term is the contribution from diagonal cyclic blocks: on a cyclic module, a balanced form has the shape $B(p,q)=\psi(pq)$, hence is symmetric because multiplication in $\F_2[t]/(\phi^a)$ is commutative.  For each off-diagonal pair of cyclic blocks, the block $M_i\times M_j$ may be chosen freely and the opposite block is then forced by symmetry.

Since
\[
        c_\phi=d\sum_i\lambda_i+2d\sum_{i<j}\min(\lambda_i,\lambda_j)
\]
and $n_\phi=d\sum_i\lambda_i$, we get $s_\phi=(n_\phi+c_\phi)/2$.  Summing over $\phi$ gives $s(\Lambda)=(n+c(\Lambda))/2$.

Now suppose $X=X^T$ has type $\Lambda$, and let $\beta(u,v)=u^Tv$ be the standard dot product.  The map
\[
        Y\longmapsto B_Y,
        \qquad
        B_Y(u,v)=\beta(Yu,v),
\]
identifies matrices commuting with $X$ with $X$-balanced bilinear forms.  Under this identification, $Y=Y^T$ is equivalent to $B_Y$ being symmetric.  Hence the symmetric centralizer dimension is $s(\Lambda)$.
\end{proof}

\subsection{Counting symmetric representatives of a fixed type}

For a rational canonical type $\Lambda$, set
\[
        S(\Lambda):=\#\{X\in\Sym_n(\F_2):X\text{ has type }\Lambda\}.
\]

\begin{lemma}\label{lem:symmetric-type-bound}
For every type $\Lambda$ of size $n$,
\[
        S(\Lambda)\leq 2^{N+s(\Lambda)-c(\Lambda)+O(n)}.
\]
More explicitly, the proof below gives
\[
        S(\Lambda)\leq 2^{N+s(\Lambda)-c(\Lambda)+n}.
\]
\end{lemma}

\begin{proof}
Fix a representative $T$ of type $\Lambda$.  If no symmetric matrix has type $\Lambda$, then $S(\Lambda)=0$ and there is nothing to prove.  Otherwise define
\[
        G_{\rm sym}(T):=\{g\in\GL_n(\F_2):g^{-1}Tg\in\Sym_n(\F_2)\}.
\]
Each symmetric representative of type $\Lambda$ is obtained from exactly $|C_{\GL_n}(T)|$ elements of $G_{\rm sym}(T)$, so
\begin{equation}\label{eq:S-lambda-quotient}
        S(\Lambda)=\frac{|G_{\rm sym}(T)|}{|C_{\GL_n}(T)|}.
\end{equation}

We first bound the numerator.  Let $\beta(u,v)=u^Tv$ be the standard dot product.  For $g\in G_{\rm sym}(T)$, define
\[
        B_g(u,v)=\beta(g^{-1}u,g^{-1}v).
\]
If $X=g^{-1}Tg$ is symmetric, then $B_g$ is a nondegenerate symmetric $T$-balanced bilinear form.  Indeed, using $g^{-1}T=Xg^{-1}$ and the self-adjointness of $X$ with respect to $\beta$,
\[
        B_g(Tu,v)=\beta(Xg^{-1}u,g^{-1}v)
        =\beta(g^{-1}u,Xg^{-1}v)=B_g(u,Tv).
\]
By Lemma~\ref{lem:balanced-dimension-lambda}, there are at most $2^{s(\Lambda)}$ possible forms $B_g$.

For a fixed bilinear form $B$, the set of $g$ satisfying $B_g=B$ has size at most the orthogonal group of $\beta$,
\[
        O(\beta)=\{U\in\GL_n(\F_2):\beta(Uu,Uv)=\beta(u,v)\text{ for all }u,v\}.
\]
Thus
\[
        |G_{\rm sym}(T)|\leq 2^{s(\Lambda)}|O(\beta)|.
\]
We use the simple bound
\begin{equation}\label{eq:orthogonal-bound}
        |O(\beta)|\leq 2^N.
\end{equation}
An element of $O(\beta)$ is determined by the images $v_1,\ldots,v_n$ of the standard basis.  These images must satisfy $\beta(v_i,v_j)=\delta_{ij}$.  When choosing $v_j$, the conditions $\beta(v_j,v_i)=0$ for $i<j$ are $j-1$ independent linear conditions.  The norm condition $\beta(v_j,v_j)=1$ is also linear over $\F_2$.  If its linear functional lies in the span of the previous ones, the equation is inconsistent with the required right-hand side; otherwise it cuts the solution space by one additional dimension.  Hence there are at most $2^{n-j}$ choices for $v_j$, and
\[
        |O(\beta)|\leq\prod_{j=1}^n2^{n-j}=2^N.
\]
Therefore
\begin{equation}\label{eq:Gsym-bound}
        |G_{\rm sym}(T)|\leq 2^{s(\Lambda)}2^N.
\end{equation}
We now bound the denominator in \eqref{eq:S-lambda-quotient}.  From \eqref{eq:gl-centralizer-formula}, every factor in the product is at least $1/2$.  The number of factors is
\[
        R(\Lambda)=\sum_\phi\sum_{i\geq1}m_i(\lambda_\phi),
\]
the total number of rational canonical blocks.  Since each block has dimension at least $1$, we have $R(\Lambda)\leq n$.  Hence
\begin{equation}\label{eq:centralizer-lower-bound}
        |C_{\GL_n}(T)|\geq 2^{c(\Lambda)-n}.
\end{equation}
Combining \eqref{eq:S-lambda-quotient}, \eqref{eq:Gsym-bound}, and \eqref{eq:centralizer-lower-bound}, we get
\[
        S(\Lambda)
        \leq \frac{2^{s(\Lambda)}2^N}{2^{c(\Lambda)-n}}
        =2^{N+s(\Lambda)-c(\Lambda)+n}.
\]
\end{proof}
\subsection{Summing over rational canonical types}

\begin{proof}[Proof of Proposition~\ref{prop:commuting-pairs}]
For $X\in\Sym_n(\F_2)$ of type $\Lambda$, Lemma~\ref{lem:balanced-dimension-lambda} gives
\[
        \#\{Y\in\Sym_n(\F_2):YX=XY\}=2^{s(\Lambda)}.
\]
Therefore
\begin{equation}\label{eq:Cn-sum-types}
        C_n=\sum_\Lambda S(\Lambda)2^{s(\Lambda)},
\end{equation}
where the sum is over rational canonical types of size $n$.

By Lemma~\ref{lem:symmetric-type-bound}, the contribution of one type is at most
\[
        S(\Lambda)2^{s(\Lambda)}
        \leq 2^{N+2s(\Lambda)-c(\Lambda)+n}.
\]
Using $s(\Lambda)=(n+c(\Lambda))/2$, this becomes
\[
        S(\Lambda)2^{s(\Lambda)}\leq 2^{N+2n}.
\]

It remains only to count the number of types.  Let $k_n$ be the number of rational canonical types of size $n$, equivalently the number of similarity classes in $M_n(\F_2)$.  Rational canonical form gives
\[
        \sum_{n\geq0}k_nu^n
        =\prod_{d\geq1}\prod_{j\geq1}(1-u^{dj})^{-I_d},
\]
where $I_d$ is the number of monic Ireducible polynomials of degree $d$ over $\F_2$.  Since $I_d\leq2^d$, this product converges at $u=1/4$.  Hence $k_n\leq C4^n=2^{O(n)}$.  Summing the bound $2^{N+2n}$ over $2^{O(n)}$ types in \eqref{eq:Cn-sum-types} proves
\[
        C_n\leq 2^{N+O(n)}.
\]
\end{proof}

\section{Proof of the main theorem}

By Corollary~\ref{cor:commuting-pairs-reduction} and Proposition~\ref{prop:commuting-pairs},
\[
        \#\{M\in\{\pm1\}^{n\times n}:MM^T=M^TM\}
        \leq 2^n\cdot2^{N+O(n)}=2^{N+O(n)}.
\]
Since the total number of sign matrices is
\[
        2^{n^2}=2^{2N+n},
\]
we get
\[
        \probb(M_nM_n^T=M_n^TM_n)
        \leq \frac{2^{N+O(n)}}{2^{2N+n}}
        =2^{-N+O(n)}.
\]
The lower bound is immediate from symmetric sign matrices: there are $2^{N+n}$ of them, and all are normal.  Thus
\[
        \probb(M_nM_n^T=M_n^TM_n)
        \geq \frac{2^{N+n}}{2^{2N+n}}=2^{-N}.
\]
Combining the two estimates proves Theorem~\ref{thm:main}.

\section{Connection with the $0$-$1$ model}\label{sec:zero-one}

The sign model and the $0$-$1$ model are related by the affine change of variables
\[
        M=2U-J,
\]
where $J$ is the all-ones matrix.  Normality is not invariant under this change of variables, but normal $0$-$1$ matrices are sent to normal sign matrices.  This gives the same sharp exponential order for the $0$-$1$ model.

\begin{corollary}\label{cor:zero-one}
Let $U_n$ be uniformly random in $\{0,1\}^{n\times n}$, and put $N=\binom n2$.  Then
\[
        \probb(U_nU_n^T=U_n^TU_n)=2^{-N+O(n)}.
\]
\end{corollary}

\begin{proof}
The lower bound is again supplied by symmetric matrices.  There are $2^{N+n}$ symmetric $0$-$1$ matrices, all of which are normal, so
\[
        \probb(U_nU_n^T=U_n^TU_n)\geq 2^{-N}.
\]
For the upper bound, let $U=(u_{ij})\in\{0,1\}^{n\times n}$ be normal.  Write
\[
        r_i=\sum_{k=1}^n u_{ik},\qquad c_i=\sum_{k=1}^n u_{ki}
\]
for the $i$th row sum and column sum.  Comparing diagonal entries in
\[
        UU^T=U^TU
\]
gives $r_i=c_i$ for every $i$.

Now set $M=2U-J$.  Expanding gives
\begin{align*}
        MM^T-M^TM
        &=4(UU^T-U^TU)  \\
        &\quad -2(UJ+JU^T-U^TJ-JU).
\end{align*}
The first term is zero because $U$ is normal.  The $(i,j)$ entry of the second bracket is
\[
        r_i+r_j-c_i-c_j,
\]
which is zero because $r_i=c_i$ for all $i$.  Hence $M$ is a normal sign matrix.

The map $U\mapsto 2U-J$ is a bijection from $\{0,1\}^{n\times n}$ to $\{\pm1\}^{n\times n}$.  Therefore the number of normal $0$-$1$ matrices is at most the number of normal sign matrices.  By Theorem~\ref{thm:main}, this number is at most $2^{N+O(n)}$.  Dividing by $2^{n^2}=2^{2N+n}$ gives
\[
        \probb(U_nU_n^T=U_n^TU_n)\leq 2^{-N+O(n)}.
\]
Together with the symmetric-matrix lower bound, this proves the corollary.
\end{proof}

\begin{remark}
The converse implication is false: a normal sign matrix need not come from a normal $0$-$1$ matrix under $M=2U-J$.  For instance,
\[
        U=\begin{pmatrix}0&0\\1&0\end{pmatrix}
\]
is not normal, while $2U-J$ is normal.  Thus the transfer above is an inclusion argument rather than an equivalence of the two normality conditions.
\end{remark}
\section*{Acknowledgements}

The author acknowledges the use of AI tools for assistance with literature searches, reference organization, and document editing.  All mathematical statements, proofs, citations, and final editorial decisions are the responsibility of the author.

\end{document}